\documentclass[a4paper,10pt]{article}
\usepackage[utf8]{inputenc}
\usepackage[T1]{fontenc}
\usepackage{amsmath,amsthm}
\usepackage{amsfonts}
\usepackage{amssymb}
\usepackage{ textcomp }
\usepackage{tikz}
\usepackage{breqn}
\usepackage{dsfont}
\usepackage{pdfpages}
\usepackage{ upgreek }
\usepackage{pifont}
\newtheorem{The}{Theorem}[section]

\newtheorem{Not}{Notation}[section]
\newtheorem{Def}{Definition}[section]

\newtheorem{Ex}{Example}[section]
\usepackage{graphicx}
\usepackage{epstopdf}
\usepackage{caption}
\usepackage{subcaption}
\usepackage{sectsty}
\title{Analysis of Stability, Bifurcation, and Chaos in Generalized Mackey-Glass Equations}
\author{ Deepa Gupta and Sachin Bhalekar }
\date{\today}

\begin{document}
	\maketitle	
	\begin{abstract}
		Mackey-Glass equation arises in the leukemia model. We generalize this equation to include fractional- order derivatives in two directions. The first generalization contains one whereas the second contains two fractional derivatives.\\
		Such generalizations improve the model because the nonlocal operators viz. fractional derivatives are more suitable for the natural systems. We present the detailed stability and bifurcation analysis of the proposed models. We observe stable orbits, periodic oscillations, and chaos in these models. The parameter space is divided into a variety of regions, viz. stable region (delay independent), unstable region, single stable region, and stability/instability switch. Furthermore, we  propose a control method for chaos in these general equations.
	\end{abstract}
	\section{Introduction}
	\noindent Insertion of the delay improves the dynamics by modeling memory in the natural systems accurately \cite{smith2011introduction,gopalsamy2013stability,bocharov2000numerical,rihan2014delay}. In the delay differential equations (DDE), the rate of change of the state at the present time depends on the state values in the past. These systems are popular among mathematicians, too, because they are infinite dimensional systems, and their stability analysis is a challenging task \cite{hale2013introduction}.  In contrast with their ODE counterparts, the scalar DDEs may exhibit chaotic oscillations \cite{ucar2003chaotic,ikeda1980optical,yalccin2007n,sprott2007simple,pham2015multi}.
	
	On the other hand, including a nonlocal operator such as fractional order derivative makes the model more realistic \cite{kilbas2006theory}. Fractional differential equations are employed to model the phenomena in Physics \cite{torvik1984appearance}, Visco-elasticity \cite{mainardi2010fractional,mainardi2008time}, Biology \cite{magin2004fractional,arshad2019fractional}, Engineering \cite{saqib2019application} and Economics \cite{fallahgoul2016fractional,tarasov2019history}. Thus, the equations involving the delay and fractional derivative are crucial \cite{bhalekar2011fractional}. Stability analysis of fractional order delay differential equations is provided in \cite{bhalekar2013stability,vcermak2015stability,bhalekar2016stability,bhalekar2019analysis,bhalekar2023can,bhalekar2024stability}.

	In their celebrated paper \cite{mackey1977oscillation}, Mackey and Glass proposed a simple scalar nonlinear delay differential equation $\dot{x}(t)=\frac{q x(t-\tau)}{1+x(t-\tau)^n}-p x(t)$ modeling physiological system. This equation models the control of hematopoiesis in chronic granulocytic leukemia patients.  The equation shows various dynamical behavior, including periodic and chaotic oscillations. 
	Oscillatory solutions of such equations and the global attractivity are presented in \cite{gopalsamy1990oscillations,zaghrout1996oscillations,gyori1999global}.
	The electronic analog of the Mackey-Glass (MG) system is designed in \cite{namajunas1995electronic,amil2015organization}. Losson et al. \cite{losson1993solution} discussed the multistability in the MG and some other DDEs.
	
	Motivated by the above discussion, we propose the generalizations to the MG model to include the Caputo fractional derivatives.\\
	The paper is organized as follows: Section \ref{sec3} present some basic definitions, theorems and the literature survey related to classical Mackey-Glass equation. Section \ref{sec5.2} describes the stability, bifurcation, chaos and chaos control in the fractional-order Mackey-Glass equation. Analysis of the generalized Mackey-Glass equation containing two fractional derivatives is given in Section \ref{sec5.3}. Validation of results are described in Section \ref{sec5.7}. Section \ref{sec5.8} concludes the results given in the paper.  
	\section{Preliminaries}\label{sec3}
	In this section, we provide some basic definitions described in the literature \cite{podlubny1998fractional,lakshmanan2011dynamics,kilbas2006theory,bhalekar2013stability,diethelm2002analysis,bhalekar2019analysis}.
	\begin{Def}[Fractional Integral]
		For any $f \in \mathcal{L}_{1}(0,m)$ the Riemann-Liouville fractional integral of order $\upmu >0$, is given by 
		\begin{equation*}
			\textit{I}^\upmu f(t)=\dfrac{1}{\Gamma(\upmu)}\int_{0}^{t}(t-\tau)^{\upmu-1}f(\tau)d\tau  , \quad   0<t<m.
		\end{equation*}
		
	\end{Def}
	\begin{Def}[Caputo Fractional Derivative]
		For $f ^{(\eta)}\in \mathcal{L}_{1}(0,m)$, $0<t<m$ and $\eta-1<\upmu\leq \eta$, $\eta \in \mathbb{N}$, the Caputo fractional derivative of function $f$ of order $\upmu$ is defined by,
		\[\textit{D}^{\upmu} f(t)=
		\begin{cases}
			\frac{d^{\eta}}{dt^{\eta}} f(t) ,\textit{ if } \quad \upmu = \eta \\ \textit{I}^{\eta-\upmu}\dfrac{d^{\eta} f(t)}{dt^{\eta}}, \textit{ if } \quad  \eta-1< \upmu < \eta. 
		\end{cases}\]
		Note that $\textit{I}^\upmu\textit{D}^\upmu f(t)=f(t)-\sum_{k=0}^{\eta-1}\dfrac{d^k f(0)}{dt^k}\dfrac{t^k}{k!}.$\\
		The Caputo derivative of a constant is zero, like the classical case.
	\end{Def}
	\begin{Def}[Equilibrium Point]
		Let us take fractional order delay differential equation:
		\begin{equation}\label{ab}
			D^\alpha{x}(t)=f(x(t),x(t-\tau)),  \quad 0<\alpha\leq 1,
		\end{equation}
		and consider the two-term fractional order generalized delay differential equation:
		\begin{equation}
			D^\alpha{x}(t)+c D^{2\alpha}x(t)=f(x(t),x(t-\tau)),  \quad 0<\alpha\leq 1,\label{aa}
		\end{equation}
		where, $\tau \textgreater0$,  $f: E\rightarrow \mathbb{R}$, $E\subseteq \mathbb{R}^2$ is open and $f\in C^1(E)$.\\
		A steady state solution is called an equilibrium point.\\
		So, $x_*$ is an equilibrium point if and only if
		\begin{equation}
			f(x_*,x_*)=0.\label{gg}
		\end{equation}
	\end{Def}
	Consider a fractional order DDE \eqref{ab} and the two term FDDE (\ref{aa}) with the initial data 
	\begin{equation}
		x(t)=\phi(t), -\tau\leq t\leq 0,\textit{ where } \phi:[-\tau,0]\rightarrow \mathbb{R}. \label{cc}
	\end{equation}
	\begin{Not}
		The solution of delay differential equation (\ref{aa}) or \eqref{ab} with initial data (\ref{cc}) is denoted by $x(t,\phi).$\\
		The norm of $\phi$ is given by 
		\begin{equation*}
			||\phi||=\sup_{-\tau\leq t \leq 0}|\phi(t)|.
		\end{equation*}
	\end{Not}
	\begin{Def}
		An equilibrium point $x_*$ is stable if for any given $\epsilon \textgreater 0$, there exists $\delta \textgreater 0$ such that \[||\phi-x_*||\textless \delta \Rightarrow |x(t,\phi)-x_*|\textless \epsilon, \quad t\geq 0.\] 
	\end{Def}
	\begin{Def}
		An equilibrium point $x_*$ is asymptotically stable if it is stable and there exists $b_0\textgreater 0$ such that \[||\phi-x_*||\textless b_0 \Rightarrow \lim_{t\longrightarrow\infty}x(t,\phi)=x_*.\]
	\end{Def}
	\begin{Def}
		An equilibrium point that is not stable is called unstable.
	\end{Def} 
	\begin{Def}[Single stable region (SSR)]
		If there exists $\tau_{*}>0$ such that $0<\tau<\tau_{*}\Longrightarrow x_*$ is stable and if $\tau>\tau_{*}\Longrightarrow x_*$ is unstable, then we say that there is a single stable region in the parameter space.
	\end{Def}
	\begin{Def}[Stability Switch (SS)] 
		We get stability switches if $\exists$ $\tau_{0*}=0$ and positive constants $\tau_{1*}$, $\tau_{2*}$, $\tau_{3*},\ldots\tau_{k*}$ such that \\
		$\tau_{2jk*}<\tau<\tau_{(2j+1)*}$  $\Longrightarrow$ $x_*$ is stable, $j=0,1,\ldots,\frac{k-1}{2}$,\\
		$\tau_{(2j+1)*}<\tau<\tau_{(2j+2)*} \Longrightarrow$ $x_*$ is unstable, $j=0,1,\ldots,\frac{k-3}{2}$\\
		and 
		$\tau>\tau_k$ $\Longrightarrow$ $x_*$ is unstable.\\
	\end{Def}	
	\begin{Def}[Instability switches (IS)]
		In this case, we will have $\tau_{0*}=0$ and positive constants $\tau_{1*}$, $\tau_{2*},	\ldots,\tau_{k*}$ such that \\
		$\tau_{2j*}<\tau<\tau_{(2j+1)*}\Longrightarrow $ $x_*$ is unstable, $j=0,1,\ldots,(k-2)/2$,\\
		$\tau_{(2j+1)*}<\tau<\tau_{(2j+2)*}\Longrightarrow $ $x_*$ is stable, $j=0,1,\ldots,(k-2)/2$\\
		and $\tau>\tau_k\Longrightarrow$ $x_*$ is unstable. 
	\end{Def}
	\begin{The}\label{11}
		\cite{bhalekar2016stability,bhalekar2023can} Consider the fractional order delay differential equation
		\begin{equation}\label{eq5.1.1.1.}
			D^{\alpha} x(t)=\gamma x(t)+\beta x(t-\tau),\quad 0<\alpha\leq1.
		\end{equation}
		\begin{itemize}
			\item[{Case 1}] If  $\beta\in (-\infty,-|\gamma|)$ then the stability region of equation \eqref{eq5.1.1.1.} in $(\tau,\gamma,\beta)$ parameter space is located between the plane $\tau=0$ and
			\begin{equation*}\label{zzzzz}
				\tau_*=\dfrac{\arccos\Bigg(\dfrac{\Bigg(\gamma\cos\Big(\dfrac{\alpha\pi}{2}\Big)+\sqrt{\beta^2-\gamma^2\sin^2\Big(\dfrac{\alpha\pi}{2}}\Big)\Bigg)\cos\dfrac{\alpha\pi}{2}-\gamma}{\beta}\Bigg)}{\Bigg(\gamma\cos\Big(\dfrac{\alpha\pi}{2}\Big)\pm\sqrt{\beta^2-\gamma^2\sin^2\Big(\dfrac{\alpha\pi}{2}\Big)}\Bigg)^{1/\alpha}}.
			\end{equation*}
			The equation undergoes Hopf bifurcation at this value.
			\item[{Case 2}] If $\beta\in (-\gamma,\infty)$ then equation \eqref{eq5.1.1.1.} is unstable for any $\tau\geq 0.$
			\item[{Case 3}] If $\beta\in (\gamma,-\gamma)$ and $\gamma<0$ then equation \eqref{eq5.1.1.1.} is stable for any $\tau\geq 0.$
		\end{itemize}
	\end{The}
	Theorem \eqref{11} plays important role in finding the stability of fractional order Mackey-Glass system.
	
	\begin{The}\label{mainthm5.1}
		\cite{bhalekar2024stability}
		Consider a two-term FDDE:
		\begin{equation}\label{eq5.1.2}
			D^\alpha x(t)+c D^{2\alpha} x(t)=a x(t)+b x(t-\tau),\quad 0<\alpha\leq1.
		\end{equation}
		The bifurcations in \eqref{eq5.1.2} are given by the curves $\Gamma_j $'s and the critical values of $c$ say  ($c_j$'s) are sketched in Figure \eqref{figure5.1.1}. We have the following stability properties:
		\begin{itemize}
			\item[(1)] If $b<0$, $c<0$, $a_1=a+b>0$ and $0<\alpha<1/2$, then we get a bifurcation value of $c$ i.e. $c_0$ (cf. Figure \eqref{figure5.1.1}$(a)$).
			\begin{itemize} 
				\item[\ding{118}]  for any $c\in(-\infty,c_0)$, we get four bifurcation values for the parameter $a_1$ at the curves $\Gamma_{13}$, $\Gamma_2$, $\Gamma_{12}$ and $\Gamma_6$ as $a_{11}$, $a_{12}$, $a_{13}$ and $a_{14}$, respectively such that-
				\begin{itemize}
					\item[$\bullet$] $a_1<a_{11}\Longrightarrow$ the system \eqref{eq5.1.2} is unstable $\forall\tau\geq0$
					\item[$\bullet$] $a_{11}<a_1<a_{12}\Longrightarrow$ IS
					\item[$\bullet$] $a_{12}<a_1<a_{13}\Longrightarrow$  SSR
					\item[$\bullet$] $a_{13}<a_1<a_{14}\Longrightarrow$ SS and
					\item[$\bullet$]	$a_1>a_{14}\Longrightarrow$ the system \eqref{eq5.1.2} is asymptotically stable $\forall\tau\geq0$.
				\end{itemize} 
				\item[\ding{118}]Similarly, for $c\in(c_0,0)$, we get the bifurcation values $a_{15}$, $a_{16}$ and $a_{17}$ for the parameter $a_1$ on the curves $\Gamma_{13}$, $\Gamma_2$ and $\Gamma_6$ respectively bifurcating the following regions:
				\begin{itemize}
					\item[$\bullet$] $a_1<a_{15}\Longrightarrow$ unstable region
					\item[$\bullet$] $a_{15}<a_1<a_{16}\Longrightarrow$ IS
					\item[$\bullet$] $a_{16}<a_1<a_{17}\Longrightarrow$ SS and
					\item[$\bullet$] $a_1>a_{17}\Longrightarrow$ stable region. 
				\end{itemize}
			\end{itemize}
			\item[(2)] Consider $b>0$, $c<0$, $a_1>0$ and $0<\alpha<1/2$, then 
			there exists bifurcation curves $\Gamma_2$, $\Gamma_{11}$ and $\Gamma_5$ (cf. Figure \eqref{figure5.1.1}$(b)$).\\
			For a fix $c$, Let the values of $a_1$ are $a_{18}$, $a_{19}$, $a_{20}$ along the bifurcations curves $\Gamma_2$, $\Gamma_{11}$ and $\Gamma_{5}$ respectively such that- 
			\begin{itemize}
				\item[$\bullet$] $a_1<a_{18}\Longrightarrow$ unstable $\forall\tau\geq0$
				\item [$\bullet$] $a_{18}<a_1<a_{19}\Longrightarrow$ we get SSR
				\item [$\bullet$] $a_{19}<a_1<a_{20}\Longrightarrow$ we get SS
				\item [$\bullet$] $a_{20}>a_1\Longrightarrow$ we get stable solution $\forall\tau\geq0$
			\end{itemize}
			\item[(3)] Now, $b>0$, $c>0$, $a_1<0$ and $1/2<\alpha<1$ then we will have two critical values of $c$ say $c_1$ and $c_2$ (cf. Figure \eqref{figure5.1.1}$(c)$) such that-
			\begin{itemize}
				\item[\ding{118}] if $c_1<c<\infty$ then we get two bifurcation values of $a_1$ on curves $\Gamma_7$ and $\Gamma_{14}$ say $a_{21}$ and $a_{22}$, respectively.
				\begin{itemize}
					\item[$\bullet$]$a_1<a_{21}\Longrightarrow$ stable $\forall\tau\geq0$
					\item[$\bullet$]$a_{21}<a_1<a_{22}\Longrightarrow$ we will obtained SS 	\item[$\bullet$]$a_{22}<a_1<0\Longrightarrow$ we will have SSR. 
				\end{itemize}
				\item[\ding{118}] if $c_2<c<c_1$ then we get only one bifurcation  value of $a_1$ i.e. $a_{23}$ on the curve $\Gamma_7$. 
				\begin{itemize}
					\item[$\bullet$] $a_1<a_{23}\Longrightarrow$ we will get stable solution $\forall\tau\geq0$.
					\item[$\bullet$] $a_{23}<a_1<0\Longrightarrow$ we have SS.
				\end{itemize}
				Note that no SSR exist in this case due to absent of the curve $\Gamma_{14}$.
				\item[\ding{118}] if $0<c<c_2$, then no bifurcation value of $a_1$ exist. So, we always have stable solution $\forall$ $\tau\geq0$. \\
				Due to not presence of the curve $\Gamma_7$ and $\Gamma_{14}$, no SSR and SS region exist in the above case.\\ 
			\end{itemize}		
			\item[(4)] When $b<0$, $1/2<\alpha<1$, $c>0$ and $a_1<0$, then we have bifurcation values of $c$ say $c_7$ and $c_5$ see. Figure \eqref{figure5.1.1}$(d)$ such that-
			\begin{itemize}
				\item[\ding{118}] if $c_7<c<\infty$, we get two bifurcation values of $a_1$ say $a_{24}$ and $a_{25}$ on the curves $\Gamma_9$ and $\Gamma_{16}$ respectively, such that
				\begin{itemize}
					\item [$\bullet$] if $a_1<a_{24}\Longrightarrow$, then we get stable solution $\forall\tau\geq0$.
					\item [$\bullet$] if $a_{24}<a_1<a_{25}\Longrightarrow$ SS region.
					\item [$\bullet$] when $a_{25}<a_1<0\Longrightarrow$ SSR region. 
				\end{itemize}
				\item[\ding{118}] when $c_5<c<c_7$ then we get three bifurcation values of $a_1$ i.e. $a_{26}$, $a_{27}$ and $a_{28}$ on the curves $a_1=2b$, $\Gamma_9$ and $\Gamma_{16}$, respectively. So, \\
				\begin{itemize}
					\item[$\bullet$] $a_1<a_{26}\Longrightarrow$ stable solution
					\item[$\bullet$]  $a_{26}<a_1<a_{27}\Longrightarrow$ SSR
					\item[$\bullet$]  $a_{27}<a_1<a_{28}\Longrightarrow$ SS
					\item[$\bullet$]  $a_{28}<a_1<0\Longrightarrow$ again SSR.
				\end{itemize}
				\item[\ding{118}] when $0<c<c_5$, we have only one bifurcation value of $a_1$ i.e. $a_1=2b$.
				\begin{itemize}
					\item[$\bullet$]	If $a_1<2b$ then we get stable solution $\forall\tau\geq0$. 
					\item[$\bullet$] If $2b<a_1<0$ then we get SSR. 
				\end{itemize} 	
			\end{itemize}
			\item[(5)] Whenever, $a_1$ and $c$ are of same sign i.e. either both are positive or both are negative then we always get unstable solution $\forall\tau\geq0$ for any $b\in\mathbb{R}$ and any $0<\alpha<1$. 
		\end{itemize}
	\end{The} 
	Note that all the values of $c_0$, $c_1$, $c_2$, $c_7$, $c_5$ and the bifurcation curves with the critical values of delay where we have change in stability are given in \cite{bhalekar2024stability}.\\~\\
	The Theorem \eqref{mainthm5.1} plays a very crucial role in deciding the stability of generalized Mackey-Glass equation.          
	
	\begin{figure}
		\subfloat[when $b<0$, $c<0$, $a_1>0$ and $0<\alpha<1/2$]{%
			\includegraphics[scale=0.8]{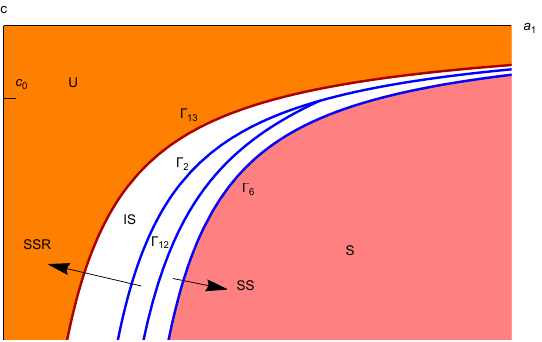}
		}\hspace{0.1cm}
		\subfloat[when $b>0$, $c<0$, $a_1>0$ and $0<\alpha<1/2$]{%
			\includegraphics[scale=0.8]{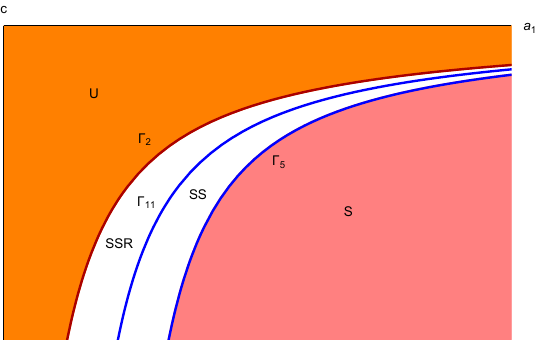}
		}\hspace{0.1 cm}
		\subfloat[when $b>0$, $c>0$, $a_1<0$ and $1/2<\alpha<1$]{%
			\includegraphics[scale=0.8]{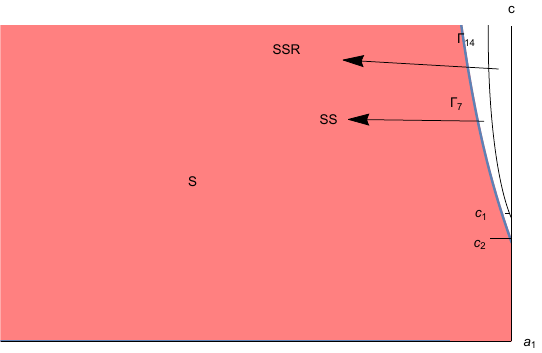}
		}\hspace{0.1cm}
		\subfloat[when $b<0$, $c>0$, $a_1<0$ and $1/2<\alpha<1$]{%
			\includegraphics[scale=0.8]{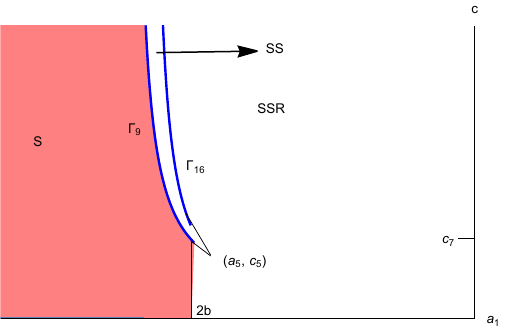}
		}
		\caption{Bifurcation diagram for different values of parameters given in equation \eqref{eq5.1.2}}
		\label{figure5.1.1}
	\end{figure}
	\subsection{Linearization near the equilibrium point $x_*$}
	\cite{bhalekar2019analysis}Assume that $x_*$ is an equilibrium point of FDDE \eqref{aa} and that $x(t)$ is a perturbed solution. Consider $\xi(t)=x(t)-x_*$. Using first order Taylor's approximation, we get 
	\begin{align*}
		D^{\alpha}\xi(t)+c D^{2\alpha}\xi(t) &= D^\alpha x(t)+cD^{2\alpha}x(t) \\&
		= f(x(t),x(t-\tau))\\&
		= f(x_*+\xi(t),x_*+\xi(t-\tau))\\&
		= f(x_*,x_*)+\partial_1 f(x_*,x_*)\xi(t)+\partial_2 f(x_*,x_*) \xi (t-\tau),
	\end{align*}
	where $\partial_1f(x_*,x_*)$ and $\partial_2f(x_*,x_*)$ are the partial derivative of $f$, with respect to first and second variable respectively.
	\begin{equation}
		\therefore D^\alpha\xi(t) +c D^{2\alpha}\xi(t)  = a\xi+ b\xi(t-\tau),\label{yy}
	\end{equation}
	where, $a=\partial_1f(x_*,x_*)$ and $b=\partial_2f(x_*,x_*)$. 
	Note that the qualitative behavior of equation \eqref{aa} is same as the qualitative behavior of linearized equation \eqref{yy} in a neighborhood of the equilibrium point $x_*$.
	
	\subsection{Mackey-Glass equation}
	The Mackey-Glass system \cite{mackey1977oscillation} is a model for the blood production in patients with leukemia. It is given by the nonlinear DDE
	\begin{equation}\label{eq5.1.1.1}
		\dot{x}(t)=-p x(t)+q \frac{x(t-\tau)}{1+x(t-\tau)^r},
	\end{equation}
	where $p$, $q$ and $r$ are the positive constants and  $x(t)$ represents the concentration of blood at time $t$. The request for blood concentration when more blood is needed is given by $x(t-\tau)$. The time for production of immature cells in the bone marrow and their maturation is the delay $\tau$.\\
	\subsubsection{Stability results}
	If $x_*$ is an equilibrium point then it must satisfy \[-p x_*+ q \frac{x_*}{1+x_*^r}=0.\]\\ 
	So, by solving it we get the equilibrium points  $x_1^*=0$ and $x_2^*=(\frac{q-p}{p})^{1/r}.$
	Note that for the existence of real equilibrium point we need $q>p$.\\
	$x_2^*$ gives a set of equilibrium  points for $r\neq1$, with the same stability properties.
	\begin{The}\cite{lakshmanan2011dynamics}
		Consider the Mackey-Glass equation \eqref{eq5.1.1.1}.
		\begin{itemize}
			\item The equilibrium point $x_1^*=0$ is unstable $\forall\tau\geq0$.
			\item  If $\frac{p}{q}\geq 1-\frac{2}{r}$ then, $x_2^*$ is asymptotically stable $\forall\tau\geq0$.
			\item  if $\frac{p}{q}<1-\frac{2}{r}$ then we get  \[\tau_*=\frac{1}{p\sqrt{\frac{(p-q)^2r^2}{q^2}+2\frac{(p-q)r}{q}}}\arccos\Big(\frac{q}{(p-q)r+q}\Big)\] such that for $0\leq\tau<\tau_*$, $x_2^*$ is asymptotically stable and for $\tau>\tau_*$ $x_2^*$ is unstable i.e. a SSR.
		\end{itemize}
	\end{The}
	Its periodic solution and the period doubling route to chaos for different range of parameter are given in \cite{junges2012intricate}. Complex dynamics like bistability, bifurcations, periodic solution and choas are discussed in \cite{duruisseaux2022bistability}. Chaos in Mackey-Glass equation can be controlled by taking suitable controls like  constant perturbation, proportional feedback control, Pyragas control and 
	state dependent delay control \cite{kiss2017controlling}.
	
	\section{Fractional order Mackey-Glass equation}\label{sec5.2}
	Let us consider the fractional order Mackey-Glass equation:
	\begin{equation}\label{eq5.1..1.1.1}
		D^\alpha x(t)=-p x(t)+q \frac{x(t-\tau)}{1+x(t-\tau)^r},\quad 0<\alpha\leq1.
	\end{equation}
	The equilibrium points of equation \eqref{eq5.1..1.1.1} are given by $x_1^*=0$ and $x_2^*=(\frac{q-p}{p})^{1/r}$.
	We have,\[f(x(t),x(t-\tau))=-p x(t)+q \frac{x(t-\tau)}{1+x(t-\tau)^r}.\] 
	So, $\gamma=\partial_1f|_{(x_*,x_*)}=-p$ and $\beta=\partial_2f|_{(x_*,x_*)}=\frac{q(1+(1-r)(x_*)^r)}{(1+(x_*)^r))^2}$.
	\subsubsection{Stability analysis of $x_1^*=0$ } 
	\begin{The}
		The equilibrium point $x_1^*=0$ is unstable $\forall\tau\geq0$.
	\end{The}
	\begin{proof}
		We have $\gamma=\partial_1f|_{x_1^*}=-p<0$ and $\beta=\partial_1f|_{x_1^*}=q$.\\
		So, we have $\beta>-\gamma$ because $q>p$. \\
		Therefore, $\beta\in(-\gamma,\infty)$ hence from Theorem \eqref{11} Case 2, we get the required result.  
	\end{proof}
	\subsubsection{Stability analysis of $x_2^*=(\frac{q-p}{p})^{1/r}$ }
	\begin{The}\label{thm5.1.1}
		If $\frac{p}{q}>1-\frac{2}{r}$, then $x_2^*$ is asymptotically stable $\forall\tau\geq0$.
	\end{The}
	\begin{proof}
		We have $\gamma=\partial_1 f(x_2^*,x_2^*)=-p<0$ and $\beta=\partial_2 f(x_2^*,x_2^*)=\frac{a(1+(1-r){x_2^*}^{r})}{(1+{x_2^*}^{r})^2}= (p-q)\frac{r p}{q}+p$.\\ 		
		So, we get \begin{equation}\label{eq5.1.4}
			\gamma+\beta= (p-q)\frac{r p}{q}<0
		\end{equation}
		Also, we have
		\begin{align*} 
			\beta-\gamma&=2p+(p-q)\frac{r p}{q}\\
			&=2p+(\frac{p}{q}-1)r p\\
			&>2 p+(1-\frac{2}{r}-1)r p=0\label{eq5.1.5}
		\end{align*}
		Hence, $\beta>\gamma$ and from equation \eqref{eq5.1.4} we obtained $\beta<-\gamma$.\\ Therefore, we get $\beta\in(\gamma,-\gamma)$. \\
		So, by the Theorem \eqref{11} Case 3, $x_2^*$ is asymptotically stable $\forall\tau\geq0$.
	\end{proof}
	\begin{The}\label{thm5.1.3}
		If $\frac{p}{q}<1-\frac{2}{r}$ then we get a critical value of delay as
		\begin{equation*}
			\tau_*=\dfrac{\arccos\Bigg(\dfrac{\Bigg(-p\cos\Big(\dfrac{\alpha\pi}{2}\Big)+\sqrt{(p+(p-q)\frac{r p}{q})^2-p^2\sin^2\Big(\dfrac{\alpha\pi}{2}}\Big)\Bigg)\cos\dfrac{\alpha\pi}{2}+p}{(p+(p-q)\frac{r p}{q})}\Bigg)}{\Bigg(-p\cos\Big(\dfrac{\alpha\pi}{2}\Big)+\sqrt{(p+(p-q)\frac{r p}{q})^2-p^2\sin^2\Big(\dfrac{\alpha\pi}{2}\Big)}\Bigg)^{1/\alpha}}.
		\end{equation*}
		such that for $0<\tau<\tau_*$, $x_2^*$ is asymptotically stable and for $\tau>\tau_*$,  $x_2^*$ is unstable. 
	\end{The}
	\begin{proof}
		When $\frac{p}{q}<1-\frac{2}{r}$ then, we get $\beta<\gamma$.\\
		Also, from Theorem \eqref{thm5.1.1}, $\beta+\gamma<0$.\\ So, we have $\beta\in(-\infty,-|\gamma|)$. \\
		So, from Theorem \eqref{11} Case 1, we get the required result.
	\end{proof}
	\subsection{Chaos in fractional order Mackey-Glass system}		
	If we take $p=1$, $q=2$, $r=10$ and $\alpha=0.9$, then we get three real equilibrium points $x_1^*=0$, $x_2^*=1$ and $x_3^*=-1$. Note that the stability properties of $x_2^*$ and $x_3^*$ will be same as the value of $\gamma$ and $\beta$ are same for both these equilibrium points. Note that here, $\frac{p}{q}<1-\frac{2}{r}$ so from Theorem \eqref{thm5.1.3}, we get $\tau_*=0.4594$. If we take the initial data as $x(t)=0.9$ for $t\in(-\tau,0)$ near $x_2^*$, we get asymptotically stable solution for $\tau=0.3<\tau_*$ (cf. Figure \eqref{figure5.7.1}($a$)) and asymptotic $1-$cycle for $\tau=0.5>\tau_*$ (cf. Figure \eqref{figure5.7.1}($b$)). If we further increase say $\tau=1.8$, we get asymptotic two-cycle (cf. Figure \eqref{figure5.7.1}($d$)). This period doubling leads to chaos (see Figure \eqref{figure5.7.1}($e$)) for $\tau=3.4$).
	\begin{figure}
		\subfloat[$x_{2}^*=1$ is stable for $\tau=0.3$]{%
			\includegraphics[scale=0.45]{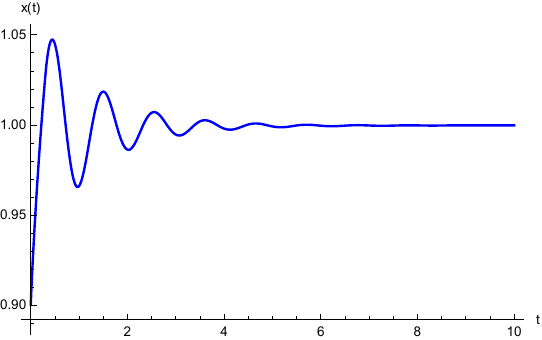}
		}\hspace{0.1 cm}
		\subfloat[$x_{2}^*=1$ is unstable for $\tau=0.5$]{%
			\includegraphics[scale=0.5]{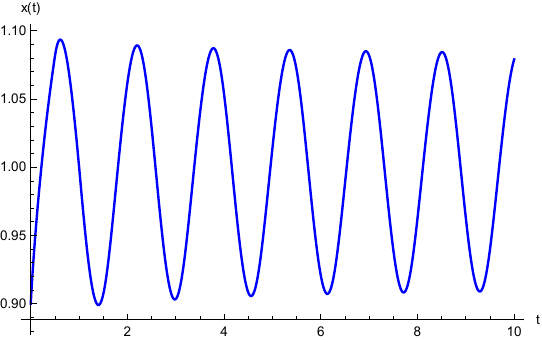}
		}\hspace{0.1 cm}
		\subfloat[Asymptotic period-two for $\tau=1$]{%
			\includegraphics[scale=0.5]{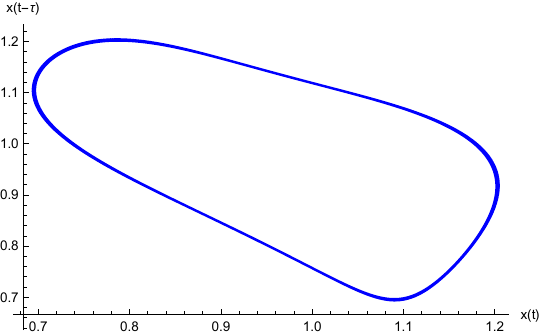}
		}\hspace{0.1 cm}
		\subfloat[Asymptotic one closed orbit for $\tau=1.8$]{%
			\includegraphics[scale=0.5]{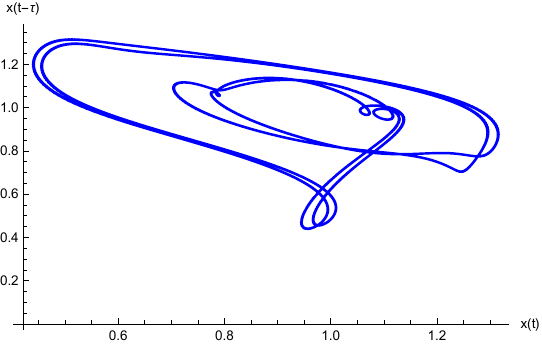}
		}\hspace{0.1 cm}
		\subfloat[ Chaotic atractor for $\tau=3.4$]{%
			\includegraphics[scale=0.5]{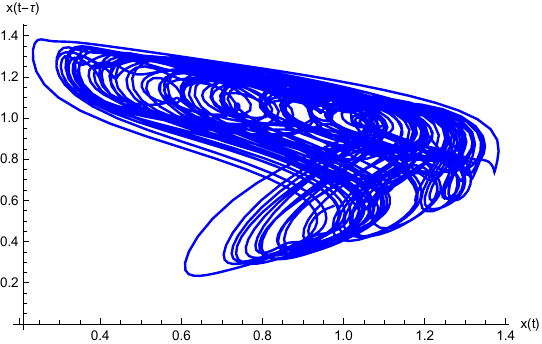}
		}
		\caption{Period doubling in fractional order Mackey-Glass equation \eqref{eq5.1..1.1.1}}
		\label{figure5.7.1}
	\end{figure}
	\subsection{Chaos control in fractional order Mackey-Glass equation}		
	Consider the controlled system:
	\begin{equation}\label{eq5.1.1.4}
		D^\alpha x(t)=-p x(t)+\frac{q x(t-\tau)}{1+x(t-\tau)^r}+k x(t).
	\end{equation}	
	We used the linear feedback control $k x(t)$ in \eqref{eq5.1.1.4}.	
	\begin{The}\label{thm5.1.2}
		The chaos in system \eqref{eq5.1.1.4} can be controlled by setting $k\in(p-q,u)$, 
		where,\[u=
		\begin{cases}
			p  ,\textit{ if } \quad r\leq2\\ p-q(1-\frac{2}{r}), \textit{ if } \quad  r>2. 
		\end{cases}\]
	\end{The}
	\begin{proof}
		Chaos in \eqref{eq5.1.1.4} can be controlled by stabilizing the unstable equilibrium point $x_2^*=(\frac{q-p-k}{p-k})^(\frac{1}{r}).$ \\
		Near $x_2^*$, the linearized equation of equation \eqref{eq5.1.1.4}, is  \[D^\alpha x(t)=-(p-k)x(t)+((p-k)-q)\frac{r(p-k)}{q}x(t-\tau).\]	So, here $\gamma=-(p-k)$ and $\beta=((p-k)-q)\frac{r(p-k)}{q}+(p-k)$.\\
		Therefore, $\gamma+\beta=((p-k)-q)\frac{r(p-k)}{q}<0$ because $p-q<k<p$.\\
		Hence,\begin{equation}\label{eq5.1.1.7}
			\beta<-\gamma.
		\end{equation}
		Also, \begin{align*}
			\beta-\gamma&=((p-k)-q)\frac{r(p-k)}{q}+2(p-k)\\
			&=(\frac{p-k}{q}-1)r (p-k)+2(p-k)\\
			&>(1-\frac{2}{r}-1)r(p-k)+2(p-k)=0\quad(\because k<p-q(1-\frac{2}{r}))
		\end{align*}
		So,
		\begin{equation}\label{eq5.1.1.6}
			\beta>\gamma.
		\end{equation}
		Hence, from equations \eqref{eq5.1.1.7} and \eqref{eq5.1.1.6}, $\beta\in(\gamma,-\gamma)$ where $\gamma<0$.\\
		Therefore, $x_2^*$ becomes stable for $k\in(p-q,u)$, $\forall\tau\geq0$. Thus, the chaos in \eqref{eq5.1.1.4} is controlled.
	\end{proof}
	Figure \eqref{figure5.7.1}$(e)$ shows chaos in equation \eqref{eq5.1..1.1.1} without control i.e. $k=0$ whereas the controlled system \eqref{eq5.1.1.4} with $k=-0.9$ and $\tau=4$ is given in Figure \eqref{fig5.1.1}. The other parameter values are $p=1$, $q=2$, $r=10$ and $\alpha=0.9$.
	\begin{figure}
		\centering
		\includegraphics[scale=0.6]{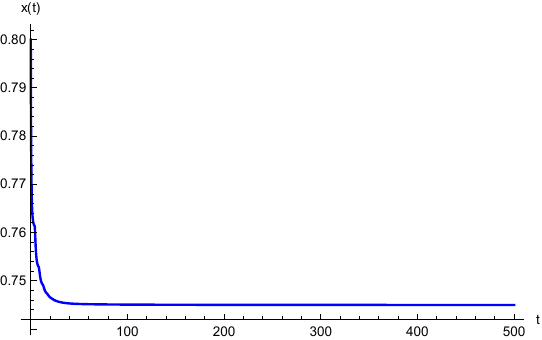}
		\caption{Chaos control in equation \eqref{eq5.1.1.4} with $k=-0.9$, $p=1$, $q=2$, $r=10$ and $\alpha=0.9$}\label{fig5.1.1}
	\end{figure}
	\section{Generalized Mackey-Glass equation}\label{sec5.3}
	Consider the generalized Mackey-Glass equation
	\begin{equation}\label{eq5.1}
		\Big(\alpha-\frac{1}{2}\Big)D^{\alpha} x(t)+d D^{2\alpha}x(t)=-p x(t)+ q \frac{x(t-\tau)}{1+x(t-\tau)^r}, \quad 0<\alpha\leq1,
	\end{equation}
	where the parameter $d\in\mathbb{R}^+$.\\
	The equilibrium points of equation \eqref{eq5.1} are  $x_1^*=0 $ and  $x_2^*=(\frac{q-p}{p})^{1/r}$. 
	\subsection{Stability analysis of the equilibrium point $x_1^*=0$ }
	Near  $x_1^*$, the linearized equation of equation \eqref{eq5.1} is given by 
	\begin{equation}\label{eq5.2}
		D^{\alpha} x(t)+\frac{d}{\alpha-\frac{1}{2}}D^{2\alpha} x(t)=\frac{-p}{\alpha-\frac{1}{2}}x(t)+\frac{q}{\alpha-\frac{1}{2}}x(t-\tau).
	\end{equation} 
	Comparing with equation \eqref{eq5.1.2}, we have $a=\frac{-p}{\alpha-1/2}$, $b=\frac{q}{\alpha-1/2}$ and $c=\frac{d}{\alpha-1/2}$.\\
	So, with the help of Theorem \eqref{mainthm5.1} we get the following results:
	\begin{The}
		The equilibrium point $x_1^*$ is unstable for all $\tau\geq0$.
	\end{The}
	\begin{proof}
		For $0<\alpha<1/2$, we have $c<0$ and $a_1=\frac{q-p}{\alpha-\frac{1}{2}}<0.$\\
		For $1/2<\alpha<1$, $c>0$ and $a_1>0$.\\
		$\therefore$ By condition ($5$) of Theorem \eqref{mainthm5.1} the equilibrium point $x_1^*$ is unstable $\forall\tau\geq0$.
	\end{proof}
	\subsection{Stability analysis of the equilibrium point $x_2^*=(\frac{q-p}{p})^{1/r}$  }
	The linearized equation of equation \eqref{eq5.1} near the equilibrium point $x_2^*$, is given by 
	\begin{equation}\label{5.1.1.1.1.1}
		D^\alpha x(t)+\frac{d}{\alpha-\frac{1}{2}}D^{2\alpha}x(t)=-\frac{p}{\alpha-\frac{1}{2}}x(t)+\frac{p+(p-q)\frac{r p}{q}}{\alpha-\frac{1}{2}} x(t-\tau).
	\end{equation}
	When we compare equation \eqref{5.1.1.1.1.1} with the equation \eqref{eq5.1.2}, we have $c=\frac{d}{\alpha-
		\frac{1}{2}}$, $a=-\frac{p}{\alpha-\frac{1}{2}}$, $b=\frac{p+(p-q)\frac{r p}{q}}{\alpha-\frac{1}{2}}$ and $a_1=\frac{(p-q)\frac{r p}{q}}{\alpha-\frac{1}{2}}$. 
	\begin{The}\label{thm5.1}
		If $r<\frac{q}{q-p}$ and 
		\begin{itemize} 
			\item[\ding{111}] $0<\alpha<\frac{1}{2}$, then
			we get a critical value of $d$ viz. $d_0=c_0(\alpha-\frac{1}{2})$ such that-
			\begin{itemize}
				\item[\ding{118}] if $d_0<d<\infty$ then we get the bifurcation values of $a_1$ on the curves $\Gamma_{13}$, $\Gamma_{2}$, $\Gamma_{12}$, $\Gamma_{6}$ i.e. $a_{11}$, $a_{12}$, $a_{13}$ and $a_{14}$ such that-
				\begin{itemize}
					\item [$\bullet$] $a_1<a_{11}\Longrightarrow$ $x_2^*$ is unstable $\forall\tau\geq0$
					\item [$\bullet$] $a_{11}<a_1<a_{12}\Longrightarrow$ we get IS
					\item [$\bullet$] $a_{12}<a_1<a_{13}\Longrightarrow$ we get SSR
					\item [$\bullet$] $a_{13}<a_1<a_{14}\Longrightarrow$ we get SS and
					\item [$\bullet$] $a_1>a_{14}\Longrightarrow$ $x_2^*$ stable $\forall\tau\geq0$.
				\end{itemize}
				\item[\ding{118}] when $0<d<d_0$, then we get a critical values of $a_1$ along the curves  $\Gamma_{13}$, $\Gamma_{2}$ and $\Gamma_{6}$ i.e. $a_{15}$, $a_{16}$ and $a_{17}$ such that-\begin{itemize}
					\item [$\bullet$] $a_1<a_{15}\Longrightarrow$ $x_2^*$ is unstable $\forall\tau\geq0$.
					\item[$\bullet$]  $a_{15}<a_1<a_{16}\Longrightarrow$ gives IS
					\item[$\bullet$]  $a_{16}<a_1<a_{17}\Longrightarrow$ gives SS and
					\item[$\bullet$]  $a_{17}>a_1\Longrightarrow$ gives stable solution $\forall\tau\geq0$.
				\end{itemize}
			\end{itemize}
			\item[\ding{111}] $\frac{1}{2}<\alpha<1$ then we get a critical values of $d$ viz. $d_1=c_1(\alpha-\frac{1}{2})$ and $d_2=c_2(\alpha-\frac{1}{2})$ such that-
			\begin{itemize}
				\item[\ding{118}] if $d_1<d<\infty$, we get a bifurcation values of $a_1$ along the curves $\Gamma_7$ and $\Gamma_{14}$ say $a_{21}$ and $a_{22}$ such that-
				\begin{itemize}
					\item[$\bullet$] $a_1<a_{21}\Longrightarrow$ $x_2^*$ is stable $\forall\tau\geq0$
					\item[$\bullet$] $a_{21}<a_1<a_{22}$ we get SS and
					\item[$\bullet$] $a_{22}<a_1<0$ we get SSR.
				\end{itemize}
				\item[\ding{118}] if $d_2<d<d_1$, then we get only one bifurcation value of $a_1$ i.e. $a_{23}$ along the curve $\Gamma_7$ so that-\begin{itemize}
					\item $a_1<a_{23}\Longrightarrow$ gives $x_2^*$ stable $\forall\tau\geq0$ and
					\item $a_{23}<a_1<0\Longrightarrow$ we get SS.
				\end{itemize}
				\item[\ding{118}] $0<d<d_2$ then $x_2^*$ is stable $\forall\tau\geq0$.
			\end{itemize}
		\end{itemize}
	\end{The}
	\begin{proof}
		\begin{itemize}
			\item [Case 1]  We have $0<\alpha<1/2$. \\
			Since, $r<\frac{q}{q-p}$ then 
			$(p-q)r>-q$ or $(p-q)\frac{ r p}{q}>-p$. \\
			So, we have $p+(p-q)\frac{r p}{q}>0$.\\
			Hence, we get $b<0$. \\
			Moreover, $a_1=\frac{(p-q)\frac{r p}{q}}{\alpha-\frac{1}{2}}>0$ and $c<0$.\\
			$\therefore$ By using $(1)$ of Theorem \eqref{mainthm5.1}, we get a critical value of $d$ say $d_0=c_0(\alpha-\frac{1}{2})$ such that if $d\in(-\infty,d_0)$, then we get the bifurcation values of $a_1$ along the curves $\Gamma_{13}$, $\Gamma_2$, $\Gamma_{12}$ and $\Gamma_6$ such that the stability of $x_2^*$ depends on these curves. If $d\in(d_0,0)$, then we get the bifurcation values of $a_1$ along the curves $\Gamma_{13}$, $\Gamma_2$, $\Gamma_{6}$ where the stability depends on this curves.    
			\item[Case 2] We have $1/2<\alpha<1$.\\
			Further $r<\frac{q}{q-p}$ $\Rightarrow p+(p-q)\frac{r p}{q}>0$.\\
			So, for $1/2<\alpha<1$, we get $b>0$, $a_1<0$ and $c>0$. So, from  Theorem \eqref{mainthm5.1}$(3)$, it completes the proof.
		\end{itemize}
	\end{proof}
	\begin{The}\label{thm5.4}
		When $r>\frac{q}{q-p}$ and
		\begin{itemize}
			\item  $0<\alpha<\frac{1}{2}$ with $a>0$, then for a fixed $d$ we have the bifurcation values of $a_1$ on the curves $\Gamma_2$, $\Gamma_{11}$ and $\Gamma_5$ say $a_{18}$, $a_{19}$ and $a_{20}$ respectively, such that-
			\begin{itemize}
				\item[$\bullet$] $a_1<a_{18}\Longrightarrow$ $x_2^*$ is unstable $\forall\tau\geq0$
				\item[$\bullet$] $a_{18}<a_1<a_{19}\Longrightarrow$ we get SSR
				\item[$\bullet$] $a_{19}<a_1<a_{20}\Longrightarrow$ we get SS and
				\item[$\bullet$] $a_{20}<a_1$ then we get $x_2^*$ stable for all $\tau\geq0$.
			\end{itemize}
			\item $1/2<\alpha<1$ with $a<0$. We obtained critical values of $d$ i.e. $d_7=c_7(\alpha-\frac{1}{2})$ and $d_5=c_5(\alpha-\frac{1}{2})$. So, that-
			\begin{itemize}
				\item[\ding{118}] if $d\in(d_7,\infty)$ we have $a_1$ values on the curves $\Gamma_9$ and $\Gamma_{16}$ i.e. $a_{24}$ and $a_{25}$ such that-
				\begin{itemize}
					\item[$\bullet$] when $a_1<a_{24}\Longrightarrow$ $x_2^*$ is stable $\forall\tau\geq0$
					\item[$\bullet$] when $a_{24}<a_1<a_{25}\Longrightarrow$ $x_2^*$ is in SS and
					\item[$\bullet$] when $a_{25}<a_1<0\Longrightarrow$ $x_2^*$ gives SSR.
				\end{itemize}
				\item[\ding{118}] when $d\in(d_5,d_7)$ then we will have three critical value of $a_1$ along the bifurcation curves $a_1=2b$, $\Gamma_9$ and $\Gamma_{16}$ say $a_{26}$, $a_{27}$ and $a_{28}$ such that-
				\begin{itemize}
					\item[$\bullet$] $a_1<a_{26}\Longrightarrow$ $x_2^*$ stable $\forall\tau\geq0$.
					\item[$\bullet$] $a_{26}<a_1<a_{27}\Longrightarrow$ $x_2^*$ is in SSR.
					\item[$\bullet$] $a_{27}<a_1<a_{28}\Longrightarrow$ $x_2^*$ is in SS.
					\item[$\bullet$] $a_{27}<a_1<0\Longrightarrow$ $x_2^*$ is in SSR.
				\end{itemize}
				\item[\ding{118}]	when $d\in(0,d_5)$ then we get only one critical value of $a_1$ along the curve $a_1=2b$ such that-
				\begin{itemize}
					\item[$\bullet$] if $a_1<2b\Longrightarrow$ $x_2^*$ gives stable solution $\forall\tau\geq0$.
					\item[$\bullet$] if $2b<a_1<0\Longrightarrow$ $x_2^*$ gives SSR.
				\end{itemize}
			\end{itemize}
		\end{itemize}
	\end{The} 		
	\begin{proof}
		\begin{itemize}
			\item[Case 1] Consider $0<\alpha<\frac{1}{2}$.\\
			Since, $r>\frac{q}{q-p}$, we have $(p-q)r<-q$.\\ Therefore, we get $(p-q)\frac{r p}{q}<-p$.\\
			Hence, $p+(p-q)\frac{r p}{q}<0$.\\
			So, we get $b>0$, $c<0$ and $a_1>0$.\\
			By employing the Theorem \eqref{mainthm5.1}$(2)$, we conclude that for a fixed $d\in(0,\infty)$, there exists a bifurcation value $a_{18}$ of $a_1$ on the curves $\Gamma_2$, $a_{19}$ on $\Gamma_{11}$ and $a_{20}$ on $\Gamma_{5}$ such that if
			$0<a_1<a_{18}$ we get $x_2^*$ unstable $\forall\tau\geq0$.\\
			Similarly, we can prove other results.
			
			\item[Case 2] Consider $\frac{1}{2}<\alpha<1$.\\
			We have $r>\frac{q}{q-p}$ then we get, $p+(p-q)\frac{r p}{q}<0$.\\
			So, $b<0$, $a_1<0$ and $c>0$.\\
			Therefore, the proof follows from Theorem \eqref{mainthm5.1} $(4)$.\\
			Hence, from Theorem \eqref{mainthm5.1} $(4)$, we get critical values of $d$ viz. $d_7$ and $d_5$ such that if we choose $d\in(d_7,\infty)$, then along the curve $\Gamma_9$ and $\Gamma_{16}$, we get critical values of $a_1$ i.e. $a_{24}$ and $a_{25}$ respectively. When $a_1<a_{24}$, $x_2^*$ is asymptotically stable $\forall\tau\geq0$.\\
			If $a_{24}<a_1<a_{25}$ then we get SS.\\ 
			On the other hand, $a_{25}<a_1<0$ gives SSR.\\
			In the similar way, we can prove other results.
		\end{itemize} 
	\end{proof} 		
	
	\subsection{Different types of chaotic attractors in the generalized Mackey-Glass equation}\label{sec5.5}
	\begin{itemize}
		\item	If we take $p=1.6$, $q=4$, $r=10$, $d=4$ and $\alpha=0.9$ then we get three real equilibrium points i.e. $0$, $-1.04138$ and $1.04138$. If we take the equilibrium point $1.04138$ then $a=-4$, $b=-20$ and $c=10$. If, $r>\frac{q}{q-p}$ then by the Theorem \eqref{thm5.4}, critical values $c_7$ and $c_5$ are $0.0550834$ and $0.0523702$ respectively for the given $b$ and $\alpha$. Since we have $c=10>c_7$, we get two bifurcation curves $\Gamma_9$ and $\Gamma_{16}$. Hence, $a_{24}=-77.2995$ and $a_{25}=-47.8727$. Since, we have chosen $a_1=a+b=-24$ hence, we are in the SSR region. So, from \cite{bhalekar2024stability} the critical value of $\tau_*=0.281952$. The $x_2^*=1.04138$ is stable for $\tau=0.2$ with the initial condition $x(t)=0.5$ and $\dot{x}(t)=0.5$ for $t\in(-\tau,0)$ as given in Figure \eqref{figure5.7}$(a)$. Double scroll chaotic attractor for $\tau=2.7$ is given in Figure \eqref{figure5.7}$(b)$.  \\ 
		
		\item On the other hand, if we take $p=0.4$, $q=0.8$, $r=10$, $d=0.08$ and $\alpha=0.9$, the three real equilibrium points are $0$, $1$ and $-1$. We get $a=-1$, $b=-4$ and $c=0.2$, if we take the equilibrium point $1$. If, $r>\frac{q}{q-p}$, then we get two critical values of $c$ viz. $c_7=0.324185$ and $c_5=0.261851$ from Theorem \eqref{thm5.4}. We are taking $c<c_5$, therefore we get only one bifurcation curve i.e. $a_1=2b$. Here, we have $a_1=-5>2b$ which is in SSR and the critical value of $\tau_*$ is $0.376184$ \cite{bhalekar2024stability}. Note that the stability behavior of the equilibrium points $1$ and $-1$ are same. So, if we choose initial condition near $-1$ then the stable solution for $\tau=0.3$ is given in Figure \eqref{figure5.7}$(c)$ with the initial data $x(t)=0.5$ and $\dot{x}(t)=0.5$ for $t\in(-\tau,0)$. The co-existing chaotic attractors for $\tau=3$ are given in Figure \eqref{figure5.7}$(d)$ with the one set of initial data $x(t)=\dot{x}(t)=0.5$ and another set of initial data as $x(t)=\dot{x}(t)=-0.5$ for $t\in(-\tau,0)$ respectively.
	\end{itemize}
	\begin{figure}
		\subfloat[$x_{2,n}^*$ is stable for $\tau=0.2$]{%
			\includegraphics[scale=0.6]{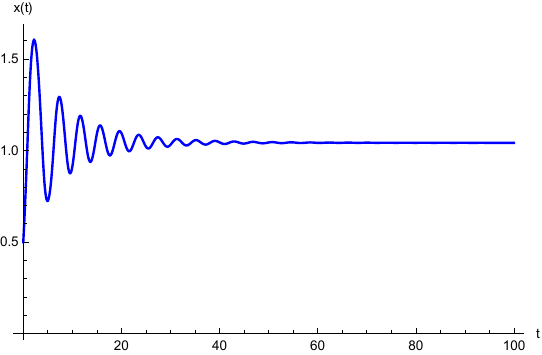}
		}\hspace{0.5 cm}
		\subfloat[double scroll attractor for $\tau=2.7$]{%
			\includegraphics[scale=0.6]{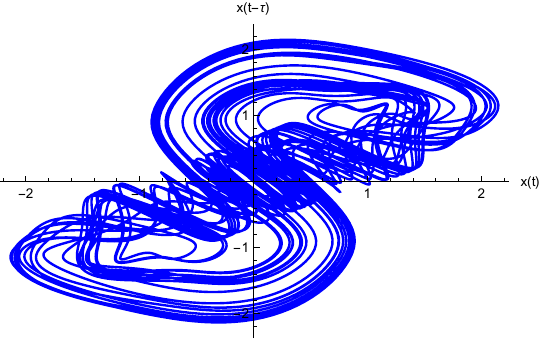}
		}\hspace{0.5 cm}
		\subfloat[stable solution for $\tau=0.3$]{%
			\includegraphics[scale=0.6]{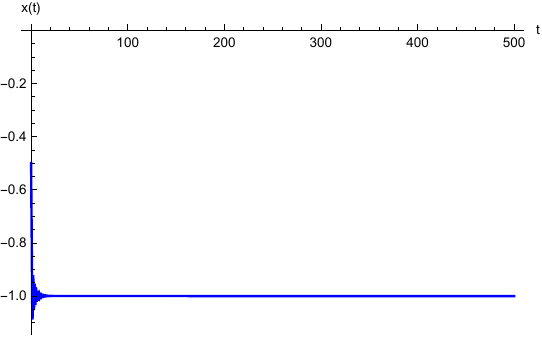}
		}\hspace{0.9 cm}
		\subfloat[co-existing chaotic atractor for $\tau=3$]{%
			\includegraphics[scale=0.6]{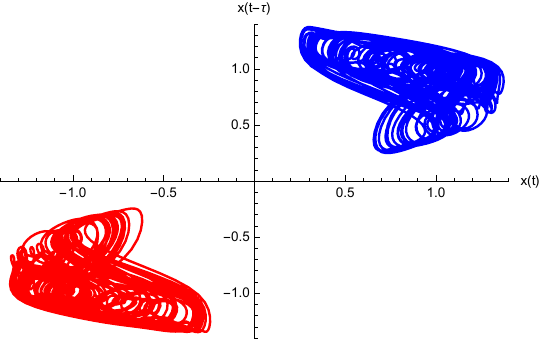}
		}
		\caption{Two types of chaotic attractors in the generalized Mackey-Glass equation}
		\label{figure5.7}
	\end{figure}
	\subsection{Chaos control}\label{sec5.6}  	
	In this section, if we consider the linear feedback control in the equation \eqref{eq5.1}. The controlled system is defined as
	\begin{equation}\label{eq5.1.1.1.1}
		\Big(\alpha-\frac{1}{2}\Big)D^{\alpha} x(t)+d D^{2\alpha}x(t)=-p x(t)+ q \frac{x(t-\tau)}{1+x(t-\tau)^r}+ k x(t),
	\end{equation}
	where $k\in\mathbb{R}$ is a control parameter.\\
	The equilibrium points of equation \eqref{eq5.1.1.1.1} are $x_1^*=0$ and $x_2^*=\Big(\frac{q-(p-k)}{p-k}\Big)^{1/r}.$\\
	The linearized equation near the equilibrium point $x_2^*$ is given by
	\begin{equation}\label{eq5.1.1.1.2}
		\Big(\alpha-\frac{1}{2}\Big)D^{\alpha} x(t)+d D^{2\alpha}x(t)=-(p-k) x(t)+((p-k)-q)\frac{r(p-k)}{q}x(t-\tau),
	\end{equation} 
	here we get $c=\frac{d}{\alpha-\frac{1}{2}}$, $a=\frac{-(p-k)}{\alpha-\frac{1}{2}}$,  $b=\frac{((p-k)-q)\frac{r(p-k)}{q}+(p-k)}{\alpha-\frac{1}{2}}$ and $a_1=\frac{((p-k)-q)\frac{r(p-k)}{q}}{\alpha-\frac{1}{2}}$.
	\begin{The}\label{thm5.1.1.1}
		If we take $1/2<\alpha<1$, $k\in(p-q(1-\frac{1}{r}),p)$ and $r>1$ then there exists a critical value of $d$ as $d_7=c_7(\alpha-\frac{1}{2})$ such that-
		\begin{itemize}
			\item when we fix $d\in(d_7,\infty)$ and choose $a_1$ less than $a_{24}$ (where $a_{24}$ is the bifurcation value of $a_1$ on the curve $\Gamma_9$) then we can control chaos in system \eqref{eq5.1.1.1.1}.
			\item when $d\in(0,d_7)$, if we choose $a_1$ less than $2b$ then chaos in system \eqref{eq5.1.1.1.1} is controlled. 
		\end{itemize} 
	\end{The} 
	\begin{proof} 
		We have $p-q<p-q(1-\frac{1}{r})<k<p$, hence $x_2^*\in\mathbb{R}$.\\
		Also, $k>p-q(1-\frac{1}{r})$ so we get $p-k-q(1-\frac{1}{r})<0$.
		\begin{align*}
			((p-k)-q)\frac{r}{q}+1<0\quad\textnormal{and}\\
			((p-k)-q)\frac{r}{q}(p-k)+(p-k)<0.
		\end{align*}
		So, we have $b<0$ and $a_1<0$.\\
		Hence, by the Theorem \eqref{mainthm5.1}$(4)$, we get critical value of $d$ as $d_7=c_7(\alpha-1/2)$ such that if we fixed $d>d_7$ and $a_1$ less than $a_{24}$ then $x_2^*$ is stable for all $\tau\geq0$.\\
		By stabilizing $x_2^*$ we can control chaos in system \eqref{eq5.1.1.1.1}.\\
		Whereas if we choose $0<d<d_7$ then for any $a_1<2b$ we get stable equilibrium point $x_2^*$ for all $\tau\geq0$.\\
		Hence, we can control chaos in the system \eqref{eq5.1.1.1.1}. 
	\end{proof}
	Figure \eqref{figure5.7}$(b)$ shows	chaos in system \eqref{eq5.1} for $p=1.6$, $q=4$, $r=10$, $d=4$ with $\alpha=0.9$ and $k=0$. For these parameters we get $k\in(-2,1.6)$ (Theorem \eqref{thm5.1.1.1}). Therefore, for $k=-1.9$ we get $a=-8.75$, $b=-2.1875$ and $a_1=-10.9375$. We get $c_7=0.50362$ hence, $d_7=0.201448$. Note that here $d>d_7$. We get $a_{24}=-6.24623$, so $a_1<a_{24}$ hence, we get $x_2^*=0.823171$ stable $\forall\tau\geq0$. For $\tau=0.1$ the stable solution is given in Figure \eqref{figure5.8}$(a)$.\\
	If we take another example where we get chaos for $k=0$ in section \eqref{sec5.5} is when $p=0.4$, $q=0.8$, $r=10$, $d=0.08$ and $\alpha=0.9$. We get $k\in(-0.32,0.4)$ from Theorem \eqref{thm5.1.1.1}. If we take $k=-0.3$, we have $a=-1.75$, $b= -0.4375$ and $a_1=-2.1875$. For these values of $\alpha$ and $b$ we get $d_7=1.00724$. Since, $d<d_7$ and $a_1<2b$, hence we get $x_2^*=0.823171$ stable $\forall\tau\geq0$. For $\tau=0.5$, the stable solution near $x_2^*$ is given in Figure \eqref{figure5.8}$(b)$.  
	\begin{figure}
		\subfloat[$x_{2}^*=0.823171$ is stable for $p=1.6$, $q=4$, $r=10$, $k=-1.9$ and $\alpha=0.9$ for $\tau=0.1$]{%
			\includegraphics[scale=0.8]{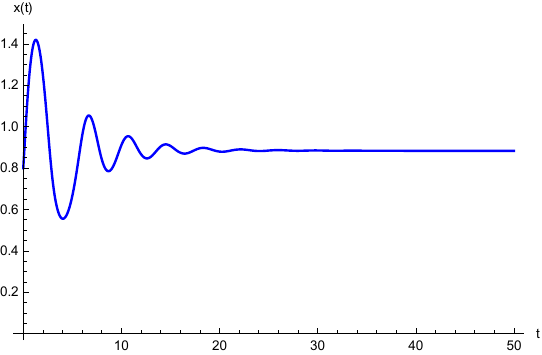}
		}\hspace{0.5cm}
		\subfloat[$x_{2}^*=0.823171$ is stable for $p=0.4$, $q=0.8$, $r=10$, $k=-0.3$ and $\alpha=0.9$ for $\tau=0.5$]{%
			\includegraphics[scale=0.8]{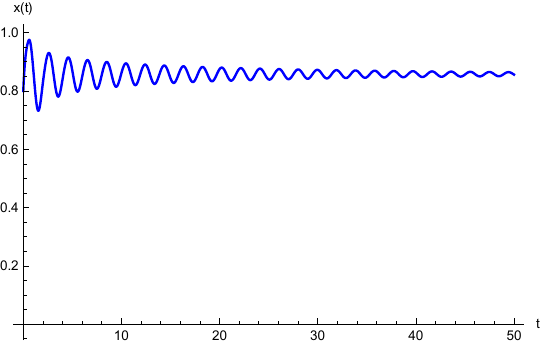}
		}
		\caption{Chaos in Section \eqref{sec5.5} is controlled by taking control parameter in the permissible range}
		\label{figure5.8}
	\end{figure}
	\section{Validation of the results}\label{sec5.7}  		
	\begin{Ex}\label{ex5.1}
		\begin{itemize}
			\item[$\bullet$] We verify some of the results given in Theorem \eqref{thm5.1}. If we take $r=2$, $p=1.4666$ and $q=2.01655$, $\alpha=0.3$, then we have $a=7.333$, $b=-3.33331$, $a_1=3.99969$, $c_0=-0.0369973$ and $d_0=0.00739946$. When we take $d=0.008$ then along the curve $\Gamma_{13}$, $a_{11}=5.9332$.  Since, $a_1<a_{11}$ so we get unstable solution near the equilibrium point $x_2^*=1$ for all $\tau\geq0$. For $\tau=0.2$ the unstable solution is given in the Figure \eqref{figure5.3.1}$(a)$.
			\item[$\bullet$] When $p=2$, $q=3$, $r=2$ and $\alpha=0.3$ the corresponding values of $a=10$, $b=-3.3333$, $a_1=6.66667$, $c_0=-0.0369971$ and $d_0=0.00739942$. When $d=0.008$, then along the curve $\Gamma_{13}$ and $\Gamma_2$ we get $a_{11}=5.9332$, $a_{12}=7.8726$, respectively. So, $a_{11}<a_1<a_{12}$ hence, we are in IS region. From \cite {bhalekar2024stability} the critical values of $\tau$ are given by two sets viz. $S1=\{0.000126443$, $0.00240816$, $0.00468988, \ldots$\} and $S2=\{0.000291275$, $0.000643295$, $0.000995316$, $\ldots\}$. The characteristic roots will shift from right half plane to the left half plane at the critical values of $\tau$ in $S1$ whereas they will move from left half plane to the right half plane at critical values of $\tau$ in $S2$. So, the equilibrium point $x_{2}*=0.707107$ is unstable for $0\leq\tau<0.000126443$, as one of the characteristic root with positive real part is $268.428 + 2997.01 I$, stable solution for $0.000126443<\tau<0.00029127$ (cf. Figure \eqref{figure5.3.1}($b$)) and again unstable solution for $\tau>0.00029127$ we get unstable solution (cf. Figure \eqref{figure5.3.1}$(c)$ for $\tau=0.0006$).
			\item[$\bullet$] If $p=2.24266$, $q=2.92869$, $r=3$ and $\alpha=0.3$ we get $a_1=7.88$, $b=-3.33333$ and $d_0=0.00739942$. If we take $d=0.008$ then $a_{12}=7.8726$ and $a_{13}=7.88691$. Hence, $a_{12}<a_1<a_{13}$ so we are in the SSR region near the equilibrium point $x_2^*=0.736006$. The critical value of delay is $\tau_*=0.000440955$. The stable equilibrium point $x_2^*$ for $\tau=0.0003$ is given in Figure \eqref{figure5.3.1}$(d)$ and one of the characteristic root with positive real part is $62.7761 + 11679 I$ for $\tau=0.0005>\tau_*$.
			\item[$\bullet$] If we choose  $p$, $q$, $r$ and $\alpha$ as $2.2666$, $4.5332$, $2$ and $0.3$ respectively, then we get $a_1=8$, $b=-3.33333$ and $d_0=0.00739942$. When we take $d=0.008$ then $a_{13}=7.88691$ and $a_{14}=8.08362$. So, $a_{13}<a_1<a_{14}$, hence we are in SS region. The critical values of $\tau$ are $S1=\{0.000493599$, $0.00102452$, $0.00155544\ldots\}$ and $S2=\{0.000798061$, $0.00160821$, $0.00241835\ldots\}$ from \cite{bhalekar2024stability}. So, we get stable solution for $\tau\in[0,0.000493599)$ (cf. Figure \eqref{figure5.3.1}$(e)$ for $\tau=0.0003$), $x_2^*=1$ will be unstable for $\tau\in(0.000493599,0.000798061)$ and one of the characteristic roots with positive real part is $35.1569 - 9988.9 I$ for $\tau=0.0006$. Again, we get stable solution for $\tau\in(0.000798061,0.00102452)$ (cf. Figure \eqref{figure5.3.1}$(f)$ for $\tau=0.0009$).
			\item[$\bullet$] If we take $p=2.46666$, $q=4.93332$, $r=2$ and $\alpha=0.3$, then we get $a_1=9$, $b=-3.3333$. Hence, $a_1>a_{14}$, so for any $\tau\geq0$, $x_2^*$ is stable. The stable solution near $x_2^*=1$ for $\tau=0.06$ is given in Figure\eqref{figure5.3.1}$(g)$.
		\end{itemize}
		\begin{figure}
			\subfloat[$x_{2}^*$ is unstable for $p=1.4666$, $q=2.01655$, $r=2$, $\alpha=0.3$, $d=0.008$ and $\tau=0.2$]{%
				\includegraphics[scale=0.4]{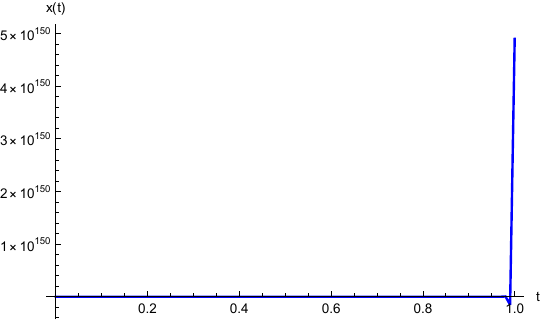}
			}\hspace{0.5 cm}
			\subfloat[$x_{2}^*$ is stable for $p=2$, $q=3$, $r=2$, $\alpha=0.3$, $d=0.008$ and $\tau=0.0002$]{%
				\includegraphics[scale=0.4]{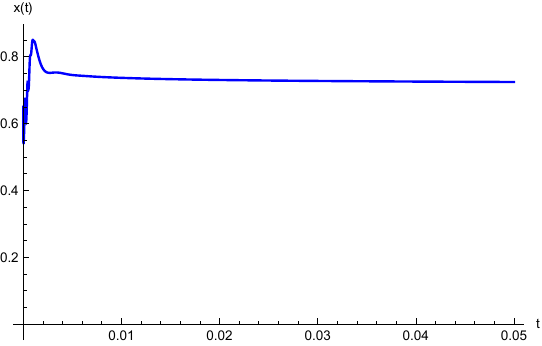}
			}\hspace{0.5 cm}
			\subfloat[$x_{2}^*$ is unstable for $p=2$, $q=3$, $r=2$, $\alpha=0.3$, $d=0.008$ and $\tau=0.0006$]{%
				\includegraphics[scale=0.4]{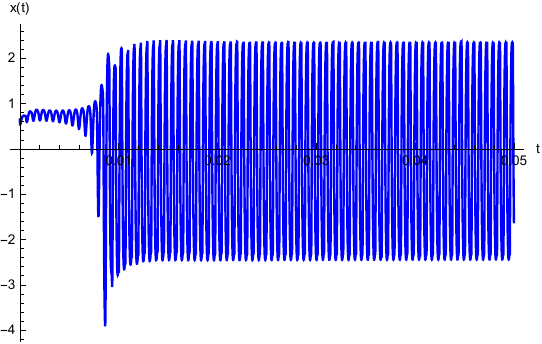}
			}\hspace{0.5 cm}
			\subfloat[$x_{2}^*$ is stable for $p=2.24266$, $q=2.92869$, $r=3$, $\alpha=0.3$, $d=0.008$ and $\tau=0.0003$]{%
				\includegraphics[scale=0.4]{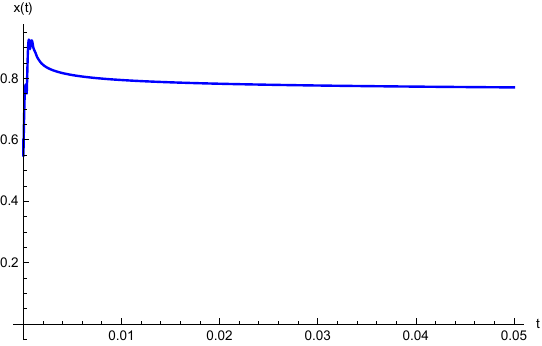}
			}	\hspace{0.5cm}
			\subfloat[$x_{2}^*$ is stable for $p=2.2666$, $q=4.5332$, $r=2$, $\alpha=0.3$, $d=0.008$ and $\tau=0.0003$]{%
				\includegraphics[scale=0.4]{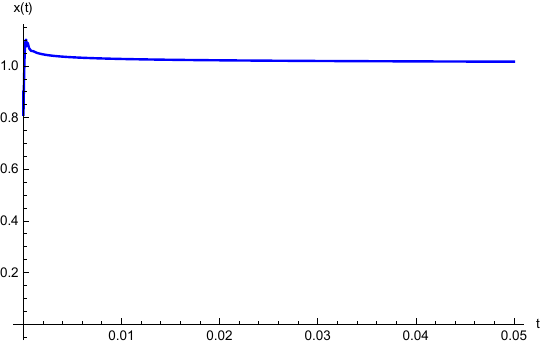}
			}
			\hspace{0.5cm}
			\subfloat[$x_{2}^*$ is stable for $p=2.2666$, $q=4.5332$, $r=2$, $\alpha=0.3$, $d=0.008$ and $\tau=0.0009$]{%
				\includegraphics[scale=0.4]{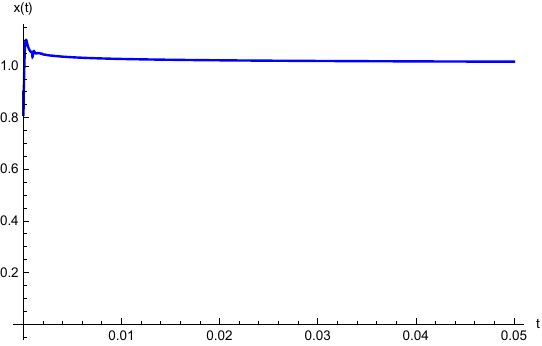}
			}
			\hspace{0.5cm}
			\subfloat[$x_{2}^*$ is stable for $p=2.4666$, $q=4.93332$, $r=2$, $\alpha=0.3$, $d=0.008$ and $\tau=0.06$]{%
				\includegraphics[scale=0.4]{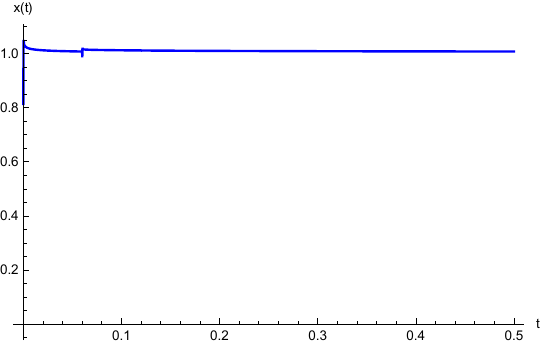}
			}	\caption{Figures for the Example \eqref{ex5.1} when $r<\frac{q}{q-p}$ and $0<\alpha<1/2$ }
			\label{figure5.3.1}
		\end{figure}
		\begin{itemize}
			\item Now, let us take $\alpha=0.8$, $p=2$, $q=9$ and $r=0.1$. So, for these values of $p$, $q$ and $r$, we get $a=-6.66667$, $b=6.14815$ and $a_1=-0.518519$. For these $b$ and $\alpha$, $c_1=0.448225$, $d_1=c_1(\alpha-1/2)=0.134468$, $c_2=0.410037$ and $d_2=0.123011$. If we take $d>d_1$, then we get two bifurcation curve $\Gamma_7$ and $\Gamma_{14}$. If $d=1$ then at the curve $\Gamma_7$ we have $a_{21}=-5.49463$ whereas at $\Gamma_{14}$, $a_{22}=-1.37386$. Since, $a_{22}<a_1<0$ so by the Theorem \eqref{thm5.1} we are in the SSR region. So, the critical value of delay as $\tau_*=2.82228$. The stable solution near $x_2^*=275855$ for $\tau=0.2$ is given in Figure \eqref{figure5.3}$(a)$ and unstable for $\tau=0.4$ is shown in Figure \eqref{figure5.3}$(b)$. 
			\item[$\bullet$] If we take $p=1$, $q=11$, $r=1$ and $\alpha=0.8$ then we get $a$, $b$ and $a_1$ as $-3.33333$, $0.30303$ and $-3.0303$	 respectively. Corresponding to these $b$ and $\alpha$ we get $c_1=9.09399$, $c_2=6.71366$, $d_1=2.7282$ and $d_2=2.0141$. If we take $d=0.9<d_2$ then we get stable solution near $x_2^*=10$ (cf. Figure \eqref{figure5.3}$(c)$ for $\tau=2$).
			\item[$\bullet$] If we take $p=5$, $q=10$, $r=0.5$ and $\alpha=0.8$ then we get $a=-16.6667$, $b=12.5$, $a_1=-4.16667$, $c_1=0.22046$ and $d_1=0.066138$. If we take $d=0.9>d_1$ then we have $a_{21}=-11.1661$ and $a_{22}=-3.41031$. Hence, $a_{21}<a_1<a_{22}$ so we are in the SS region. One set of critical values of $\tau$ are $S1=\{1.76109$, $3.94814$, $6.13519\ldots$\} and $S2=\{3.21034$, $6.70077$, $10.1912\ldots$\}. $x_2^*=1$ is stable for $\tau\in[0,1.76109)$, unstable for $\tau\in(1.76109,3.21034)$, again stable for $\tau\in(3.21034,3.94814)$ and so on. The stable solution for $\tau=1.6$ and $\tau=3.5$ are given in Figures \eqref{figure5.3}$(d)$ and \eqref{figure5.3}$(f)$. The unstable solution for $\tau=3$ is given in Figure \eqref{figure5.3}$(e)$. 
		\end{itemize}
	\end{Ex}  		
	\begin{figure}
		\subfloat[$x_{2}^*$ is stable for $p=2$, $q=9$, $r=0.1$, $\alpha=0.8$, $d=1$ and $\tau=0.2$]{%
			\includegraphics[scale=0.35]{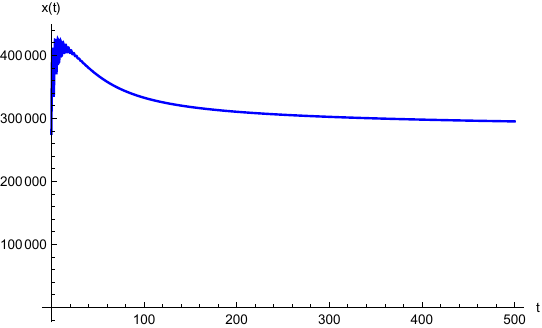}
		}	\hspace{0.5cm}
		\subfloat[$x_{2}^*$ is unstable for $p=2$, $q=9$, $r=0.1$, $\alpha=0.8$, $d=1$ and $\tau=0.4$]{%
			\includegraphics[scale=0.35]{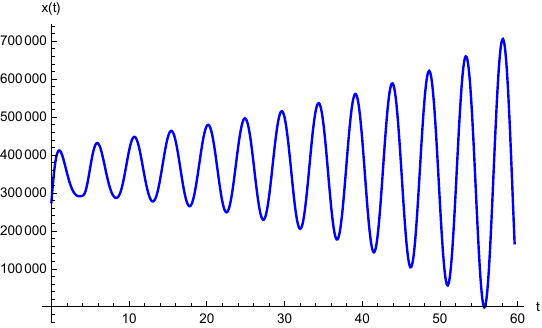}
		}	\hspace{0.5cm}
		\subfloat[$x_{2}^*$ is stable for $p=1$, $q=11$, $r=1$, $\alpha=0.8$, $d=0.9$ and $\tau=2$]{%
			\includegraphics[scale=0.35]{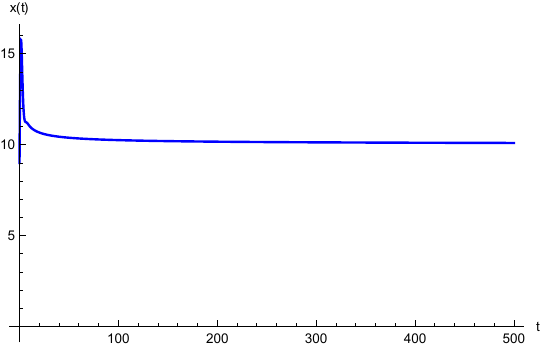}
		}	\hspace{0.5cm}
		\subfloat[$x_{2}^*$ is stable for $p=5$, $q=10$, $r=0.5$, $\alpha=0.8$, $d=0.9$ and $\tau=1.6$]{%
			\includegraphics[scale=0.35]{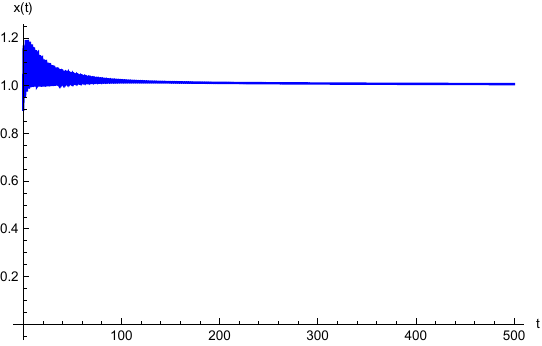}
		}
		\subfloat[$x_{2}^*$ is unstable for $p=5$, $q=10$, $r=0.5$, $\alpha=0.8$, $d=0.9$ and $\tau=3$]{%
			\includegraphics[scale=0.35]{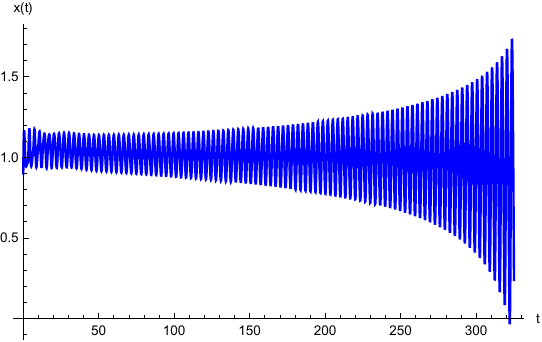}
		}	\hspace{0.5cm}
		\subfloat[$x_{2}^*$ is unstable for $p=5$, $q=10$, $r=0.5$, $\alpha=0.8$, $d=0.9$ and $\tau=3.5$]{%
			\includegraphics[scale=0.35]{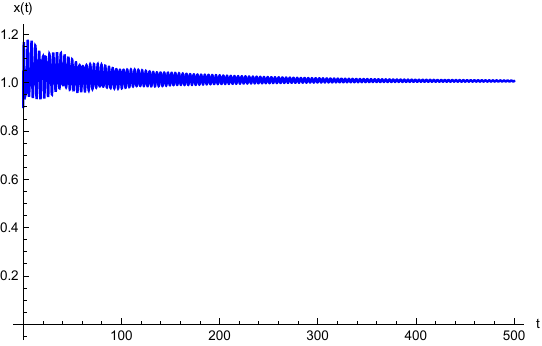}
		}
		\caption{Figures for the Example \eqref{ex5.1} when $r<\frac{q}{q-p}$ and $1/2<\alpha<1$ }
		\label{figure5.3}
	\end{figure}
	\section{Conclusion}\label{sec5.8}
	We proposed two generalizations to the Mackey-Glass equation defined by equation \eqref{eq5.1..1.1.1} and equation \eqref{eq5.1}.
	
	These equations reduce to the classical case for $\alpha=1$ and $\alpha=1/2$, $d=1$, respectively. The equation \eqref{eq5.1..1.1.1} shows the qualitative behaviour similar to the classical Mackey-Glass equation.\\
	On the other hand, as the equation \eqref{eq5.1} contains two fractional order derivatives, we observed various stability behaviors such as stability switch, single stable region, delay independent stable and unstable regions in the parameter plane. \\
	All these models possess identical equilibrium points viz. $x_1^*=0$ and  $x_2^*=(\frac{q-p}{p})^{1/r}$. The equilibrium point $x_1^*$ is always unstable whereas a richer dynamics is observed for the set of equilibrium points $x_2^*$.\\
	For the classical case and for the equation \eqref{eq5.1..1.1.1}, $x_2^*$ is asymptotically stable for smaller values of delay and becomes unstable after the critical value $\tau_*$. \\
	For the equation \eqref{eq5.1}, $x_2^*$ exhibits various bifurcations depending on the parameters $d$, $p$, $q$, $r$, $\alpha$ and $\tau$.\\
	The equation \eqref{eq5.1..1.1.1} shows one-scroll chaotic attractor only, whereas the equation \eqref{eq5.1} shows double-scroll attractor as well. \\
	The chaos in both these equations can be controlled by using a simple linear feedback control.  
	\section*{Acknowledgments}
	S. Bhalekar acknowledges the University of Hyderabad for Institute of Eminence-Professional Development Fund (IoE-PDF) by MHRD (F11/9/2019-U3(A)).
	D. Gupta thanks University Grants Commission for financial support (No. F. 82-44/2020(SA-III)).
\bibliographystyle{elsarticle-num}
\nocite{*}
\bibliography{paperref}
	\end{document}